\documentclass{amsart}
\usepackage[foot]{amsaddr}

\usepackage[utf8]{inputenc}
\usepackage[T1]{fontenc}
\usepackage[english]{babel}
\usepackage{amssymb,latexsym}
\usepackage{amsmath}
\usepackage{amsthm}
\usepackage{graphicx}
\usepackage{color}
\usepackage[shortlabels]{enumitem}
\usepackage[colorlinks=true,citecolor={blue},draft=false, urlcolor={black}]{hyperref}

\usepackage{cleveref}

\usepackage{float}
\usepackage{todonotes}

\usepackage{mathtools}

\def\N{\mathbb N}
\def\A{\mathcal A}
\def\B{\mathcal B}
\def\L{\mathcal L}
\def\uu{\mathbf u}
\def\vv{\mathbf v}

\def\Der{\mathrm{Der}_f}
\def\dd{\mathbf d}

\def\N{\mathbb N}
\def\A{\mathcal A}
\def\B{\mathcal B}

\newtheorem{thm}{Theorem}
\newtheorem{coro}[thm]{Corollary}
\newtheorem{lem}[thm]{Lemma}

\newtheorem{obs}[thm]{Observation}
\newtheorem{prop}[thm]{Proposition}
\newtheorem{defi}[thm]{Definition}
\theoremstyle{remark}
\newtheorem{example}[thm]{Example}
\newtheorem{remark}[thm]{Remark}

\crefname{thm}{theorem}{theorems}
\crefname{theorem}{theorem}{theorems}
\crefname{coro}{corollary}{corollaries}
\crefname{example}{example}{examples}
\crefname{lem}{lemma}{lemmas}
\crefname{lmm}{lemma}{lemmas}
\crefname{claim}{claim}{claims}
\crefname{obs}{observation}{observations}
\crefname{proposition}{proposition}{propositions}
\crefname{prop}{proposition}{propositions}
\crefname{defi}{definition}{definitions}

\crefname{example}{example}{examples}

\begin{document}

\title{On substitutions closed under derivation: examples}

\author{Václav Košík \and Štěpán Starosta}

\address[V. Košík]{Department of Mathematics, Faculty of Nuclear Sciences and Physical Engineering, Czech Technical University in Prague, Břehová 7, 115 19, Prague 1, Czech Republic}

\address[Š. Starosta]{Department of Applied Mathematics, Faculty of Information Technology, Czech Technical University in Prague, Thákurova 9, 160 00, Prague 6, Czech Republic}
%
%
%

%

\begin{abstract}
We study infinite words fixed by a morphism and their derived words.
A derived word is a coding of return words to a factor.
We exhibit two examples of sets of morphisms which are closed under derivation --- any derived word with respect to any factor of the fixed point is again fixed by a morphism from this set.
The first example involves standard episturmian morphisms, and the second concerns the period doubling morphism.
\end{abstract}

\maketitle              

\section{Introduction}

In 1998 Fabien Durand characterized primitive substitutive sequences, i.e., morphic images of fixed points of primitive substitutions.
A crucial role in his characterization is played by the notion ``derived word''.
Any primitive substitutive sequence $\uu$ is uniformly recurrent, i.e. for each factor $w$, the distances between consecutive occurrences of $w$ in $\uu$ are bounded.
Or equivalently, there are only finitely many gaps between neighbouring occurrences of $w$.
An infinite word coding ordering of these gaps (seen as finite words) is called the derived word to $w$ in $\uu$ and is denoted ${\bf d_{\bf u}}(w)$.

The mentioned main result of \cite{Durand98} says that a uniformly recurrent word is primitive substitutive if and only if the set of derived words to all prefixes of $\uu$ is finite.
If moreover, $\uu$ is fixed by a primitive substitution, then the derived word to a prefix $w$ of $\uu$ is fixed by a primitive substitution as well.
In other words, given any primitive substitution $\varphi$, there exists a finite list $L = \{\varphi_1, \varphi_2, \ldots,\varphi_k\}$ of primitive substitutions such that for each prefix $w$ of $\uu$, the fixed point of $\varphi$, the derived word ${\bf d_{\bf u}}(w)$ is fixed by a substitution $\varphi_i $ from $L$.
An algorithm which to a given Sturmian substitution creates such list $L$ is described in \cite{KlMePeSt18}.

On the other hand, if $w$ is a non-prefix factor of $\uu$, then it seems that ${\bf d_{\bf u}}(w)$ is fixed by a substitution only exceptionally.
In \cite{KlPeSt19}, this phenomenon is studied for fixed points of Sturmian substitutions.
For this purpose, the following new notion has been introduced:

\begin{defi} \label{DEF}
A finite non-empty set $M$ of primitive substitutions is said to be closed under derivation if the derived word ${\bf d_{\bf u}}(w)$ to any factor $w$ of any fixed point ${\bf u}$ of $\varphi \in M$ is fixed (after a suitable renaming of letters) by a substitution $\psi \in M$.
A primitive substitution $\xi$ is said to be closeable under derivation if it belongs to a set $M$ closed under derivation.
\end{defi}
Sturmian substitutions closeable under derivation are characterized in \cite{KlPeSt19}.
The aim of this contribution is to provide two new examples of sets $M$ closed under derivation.

In our first example, in \Cref{sec:ex1}, the set $M$ is a finite subset of the monoid of episturmian morphisms.
In this case, all substitutions in $M$ act on the same alphabet.
In our second example, in \Cref{sec:ex2}, the substitutions in $M$ act on alphabets with distinct cardinality.
An inspiration for the second example comes from a recent result by Huang and Wen in \cite{HuangWen17}, where a curious property of the period doubling substitution $\psi(a)=ab$ and $\psi(b)=aa$ was observed.

\section{Preliminaries}

Let $\A$ denote an \emph{alphabet} --- a finite set of symbols.
A \emph{word} over $\A$ is a finite sequence $u= u_1u_2\cdots u_n$ where $u_i \in \A$ for all $i=1,2, \ldots, n$.
The \emph{length} of the word $u$ is denoted $|u|$ and is equal to $n$.
The set of all words over $\A$ together with the operation concatenation forms a free monoid $\A^*$, its neutral element is the \emph{empty word} $\varepsilon$.
If $u=pws\in \A^* $, then $w$ is a \emph{factor} of $u$, $p$ is a \emph{prefix} of $u$, and $s$ is a \emph{suffix} of $u$.
For $w=uv$, we write $u=wv^{-1}$ and $v=u^{-1}w$.

A \emph{morphism} $\varphi$ is a mapping $\varphi : \A^* \mapsto \B^*$ such that $ \varphi(uv) =\varphi(u)\varphi(v)$ for all $u,v \in \A^*$.
A morphism $\varphi : \A^* \mapsto \A^*$ is called \emph{primitive} if there exists an iteration $k\in \N $ such that for any pair $a,b$ of letters from $\A$, the letter $a$ occurs in $\varphi^k(b)$.
In accordance with Durand's terminology, a morphism $\varphi$ is a \emph{substitution} if there exist $a\in \A$ and $w\in \A^*, w\neq \varepsilon$ such that $\varphi(a)=aw$ and $|\varphi^n(a)|$ tends to infinity with growing $n$.

An \emph{infinite word} over $\A$ is an infinite sequence ${\bf u }= u_0 u_1u_2\cdots$ from $\A^{\N}$.
A finite word $w$ of length $n$ is a factor of ${\bf u }$ if there exists an index $i\in \N$, such that $w=u_iu_{i+1}\cdots u_{n-1}$.
The index $i$ is called an \emph{occurrence} of $w$ in $\uu$.
The set of all factors of ${\bf u }$ is denoted by $\L(\uu)$.
If each factor $w$ of ${\bf u }$ has infinitely many occurrences, then ${\bf u }$ is \emph{recurrent}.
A \emph{return word} to $w$ in ${\bf u }$ is a factor $r=u_iu_{i+1}\cdots u_{j-1}$, where $i<j$ are two consecutive occurrences of $w$ in ${\bf u }$.
The word $rw$ is called a \emph{complete return word} to $w$ in $\uu $ and obviously, $rw$ is a factor of $\uu$.
The set of all return words to $w$ in ${\bf u }$ is denoted by $\mathcal{R}_{\bf u }(w)$.
If the set $\mathcal{R}_{\bf u }(w)$ is finite, say $\mathcal{R}_{\bf u }(w)= \{r_0,r_1,\ldots, r_{k-1}\}$, then ${\bf u }$ can be written as a concatenation ${\bf u } = pr_{i_0}r_{i_1}r_{i_2}\cdots$, where $p$ is the prefix of {\bf u } such that the factor $w$ occurs in $pw$ exactly once.
The infinite word $ i_0 i_1i_2\cdots$ over the alphabet $\{0,1,2,\ldots, k-1\}$ is the \emph{derived word to $w$ in ${\bf u }$} and is denoted $\dd_{\bf u }(w)$.
A recurrent infinite word ${\bf u }$ is \emph{uniformly recurrent} if the set $\mathcal{R}_{\bf u }(w)$ is finite for all $w\in\L(\uu)$.

The domain of a morphism $\varphi: \A^* \mapsto \B^*$ is naturally extended to $\A^{\N}$ by putting $\varphi({\bf u })= \varphi({u_0 u_1u_2\cdots})= \varphi(u_0) \varphi(u_1)\varphi(u_2)\cdots $.
A word ${\bf u }$ is \emph{purely substitutive} if there exists a substitution $\varphi$ over $\A$ such that ${\bf u }=\varphi({\bf u })$, i.e. ${\bf u }$ is a fixed point of $\varphi$.
A word $\vv$ over $\B$ is \emph{substitutive} if $\vv=\psi(\uu)$, where $\psi: \A^* \mapsto \B^*$ is a morphism and ${\bf u }$ is a purely substitutive word.
If $\uu$ is fixed by a primitive substitution, then $\vv$ is \emph{primitive substitutive}.
A well known fact is that a primitive substitutive word is uniformly recurrent (c.f. \cite{Durand98}).

\section{The set of derived words to factors of an infinite word}

In this section we list several simple properties of the set
\[
\Der(\uu) = \{\dd_{\bf u }(w) \colon w \in \L(\uu) \}.
\]
First, we show that only some special factors need to be examined to describe $\Der(\uu)$.
A letter $a\in \A$ is a \emph{right extension} of $w \in \L(\uu)$ if $wa \in \L(\uu)$.
Note that any factor of $\uu$ has at least one right extension.
A factor $w \in \L(\uu)$ is \emph{right special} if it has at least two distinct right extensions.
Analogously, we define \emph{left special}.

A factor which is simultaneously right and left special is \emph{bispecial}.
\begin{prop}\label{claim1}
Let $\uu$ be an infinite recurrent word over $\A$ and $w \in \L(\uu)$.
\begin{enumerate}[(1)]
\item \label{it:cl1:1} If $w$ is not left special, then $\mathcal{R}_{\bf u }(aw) =a\mathcal{R}_{\bf u }(w)a^{-1}$, where $ a \in \A $ is the unique left extension of $w$.
Moreover, if $w$ is not a prefix of $\uu$, then
$ \dd_{\bf u }(aw)=\dd_{\bf u }(w)$.
\item \label{it:cl1:2} If $w$ is not right special, then $\mathcal{R}_{\bf u }(wa) =\mathcal{R}_{\bf u }(w)$
 and $ \dd_{\bf u }(wa)=\dd_{\bf u }(w)$, where $ a \in \A $ is the unique right extension of $w$.
\end{enumerate}
\end{prop}

\begin{proof} \Cref{it:cl1:1}:
 First assume that $w$ is not left special and $w$ is not a prefix of $\uu$.
 The integer $i$ is an occurrence of $w$ in $\uu$ if and only if $i-1$ is an occurrence of $aw$ in $\uu$.
 Consequently, $r \in \mathcal{R}_{\bf u }(w)$ if and only if $ara^{-1}\in \mathcal{R}_{\bf u }(w)$ and the ordering of the return words to $w$ in $\uu$ and the ordering the return words to $aw$ in $\uu$ coincide.

 Let $0$ be an occurrence of $w$, i.e., $w$ is a prefix of $\uu$.
 Then a return word $r$ to $w$ and $rw$ have an occurrence $0$.
 We have to show that even for such $r$ the word $ara^{-1}$ belongs to $\mathcal{R}_{\bf u }(aw)$.
 Indeed, the word $\uu$ is recurrent and thus $rw$ has an occurrence $j>0$.
 As $w$ is always preceded by the letter $a$ and $a$ is a suffix of $r$ we can conclude that $ara^{-1}$ is a return word to $aw$ in $\uu$.

 \Cref{it:cl1:2}: The proof is analogous.
 \end{proof}

We formulate a straightforward corollary of \Cref{claim1}.

\begin{prop}\label{claim3}
Let $\uu$ be an infinite recurrent word over $\A$.
We have
\[
\begin{split}
\Der(\uu) = & \left \{\dd_{\bf u }(w) \colon w \text{ is a right special prefix of } \uu \right \} \\ & \cup  \left \{\dd_{\bf u }(w) \colon w \text{ is a bispecial factor of } \uu \right \}.
\end{split}
\]
\end{prop}

The following claim is taken from Durand's article.
His proof is constructive and provides an algorithm for finding a suitable morphism.

\begin{prop}[\cite{Durand98}] \label{11}
Let $\uu\in \A^{\N}$ be a fixed point of a primitive morphism $\varphi$ and $w$ be a prefix of $ \uu$.
The word $\dd_{\bf u }(w)$ is fixed by a primitive morphism as well.
\end{prop}
 \begin{proof}[Sketch of the proof]
We do not repeat the whole proof, we only describe the construction of a primitive morphism fixing $\dd_{\bf u }(w)$.

 Let $r_0,r_1, \ldots, r_{k-1}$ be the return words to $w$. Since $\uu $ is fixed by $\varphi$, the image $\varphi(w)$ has a prefix $w$ and thus $\varphi(r_iw)$ has a prefix $\varphi(r_i)w$. As $w$ is a prefix and a suffix of $\varphi(r_i)w$, the factor $\varphi(r_i)$ is concatenation of several return words to $w$, i.e. we can find unique indices $s_1,s_2, \ldots, s_{\ell_i} \in \{ 0,1,\ldots, k-1\}$ such that
 $\varphi(r_i) = r_{s_1}r_{s_2}\cdots r_{s_{\ell_i}}$. It is easy to check that the morphism given by
 \begin{equation*}\label{subst}
 \delta:\quad i\mapsto s_1s_2\cdots s_{\ell_i} \quad \text{for each } i\in \{ 0,1,\ldots, k-1\}
 \end{equation*}
is primitive and fixes $\dd_{\bf u }(w)$. All details can be found in \cite{Durand98}.
\end{proof}

\begin{prop}\label{der_of_factor}
Let $\uu\in \A^{\N}$ be a fixed point of a primitive morphism $\varphi$ and $w \in \L(\uu)$.
The word $\dd_{\bf u }(w)$ is primitive substitutive.
\end{prop}

\begin{proof}
Let $pw$ be the shortest prefix of $\uu$ containing the factor $w$.
Denote by $r_0,r_1, \ldots, r_{k-1}$ the return words to $pw$ and by $\tilde{r}_0,\tilde{r}_1, \ldots, \tilde{r}_{j -1}$ the return words to $w$.
As $w$ is a prefix and a suffix of the factor $ p^{-1}r_ipw$, the word $ p^{-1}r_ip $  can be written as concatenation of the return words to $w$, i.e. $ p^{-1}r_ip= \tilde{r}_{s_1}\tilde{r}_{s_2}\cdots \tilde{r}_{s_{\ell_i}}$ for some indices $ s_1,s_2, \ldots, s_{\ell_i} \in \{0,1,\ldots, j-1\}$.
 Define a morphism $\psi: \{ 0,1,\ldots, k-1\}^* \mapsto \{ 0,1,\ldots, j-1\}^* $ by
 \begin{equation*}\label{subst2}
 \psi :\quad i\mapsto s_1s_2\cdots s_{\ell_i} \quad \text{for each } i\in \{ 0,1,\ldots, k-1\}.
 \end{equation*}
 It follows that $\dd_{\bf u }(w)=\psi\bigl(\dd_{\bf u }(pw)\bigr) $.
 By \Cref{11}, $\dd_{\bf u }(pw)$ is fixed by a primitive substitution.
 \end{proof}

We finish this section by an example.

\begin{example}\label{example factor a}
Recall the period doubling substitution
 \begin{equation*}\label{perioddoubling}
            \psi(a)=ab \quad \text{ and } \quad \psi(b)=aa,
\end{equation*}
and its fixed point
$$ \mathbf{z} = abaaabababaaabaaabaa\ldots. $$

\begin{itemize}
\item Any occurrence of the letter $b$ is preceded and followed by the letter $a$, therefore $b$ is neither right nor left special. By \Cref{claim1}, $$\dd_{\bf \mathbf{z}}(b) = \dd_{\bf \mathbf{z}}(ab) =\dd_{\bf \mathbf{z}}(aba).$$

 \item There are two return words to $a$ in ${\bf \mathbf{z}}$, namely $r_0=ab$ and $r_1=a$.
 We can write
 $$ \mathbf{z} = r_0r_1r_1r_0r_0r_0r_1r_1r_0r_1r_1r_0r_1\ldots \quad \text{and thus} \quad \dd_{\bf \mathbf{z}}(a)=
 0110001101101\ldots. $$
 The word $ \dd_{\bf \mathbf{z}}(a)$ is fixed by a substitution.
 To find it, we compute
 $$\psi(r_0)= \psi(ab) = abaa= r_0r_1r_1 \quad \text{and } \quad \psi(r_1)= \psi(a) = ab= r_0 .$$
It follows from the proof of \Cref{11} that $ \dd_{\bf \mathbf{z}}(a)$ is fixed by the substitution $\xi$ determined by
\begin{equation*}\label{ka} \xi(0)=011 \quad \text{ and } \quad \xi(1)=0.
\end{equation*}
\end{itemize}
 \end{example}

\section{Example 1: Standard episturmian morphisms} \label{sec:ex1}
Let us recall the definition of standard Arnoux--Rauzy words and known results on morphisms fixing these words.
All mentioned facts and further results can be found in the survey \cite{GlJu}. 

\begin{defi}
An infinite word $\uu \in \A^\N$ is \emph{Arnoux--Rauzy} if
\begin{enumerate}
\item  $\uu$ has exactly one right special factor of each length;
\item $wa \in \mathcal{L}(\uu)$ for every right special factor $w$ of $\uu$ and every letter $a \in \A$;
    \item  $\mathcal{L}(\uu)$  is closed under reversal, i.e. $v_1v_2\cdots v_n \in \mathcal{L}(\uu)$  implies  $v_nv_{n-1}\cdots v_1 \in \mathcal{L}(\uu)$.
\end{enumerate}
An Arnoux--Rauzy word $\uu$ is \emph{standard} if each of its prefixes is a left special factor of $\uu$.
\end{defi}

The Arnoux--Rauzy words represent a  generalization of Sturmian words to multiliteral alphabets and share many properties with Sturmian words.
A property which is important for a description of their derived words is that Arnoux--Rauzy words  are aperiodic and by \cite{DrJuPi} they are also uniformly recurrent.
Let $\mathcal{M}_\A $ denote the monoid generated by \emph{standard episturmian morphisms} $L_a$ defined for every $a \in \A$ as follows:
$$
L_a:
\begin{cases}
a \to a \\
b \to ab \ \ \ \text{for all } b\neq a
\end{cases}.
$$
To abbreviate the notation of elements of the monoid $\mathcal{M}_\A $, we put
$$
L_z= L_{z_1}\circ L_{z_2}\circ \cdots \circ  L_{z_n}\quad \text{for\ } \ z = z_1z_2z\cdots z_n\in {\A}^*.
$$
 A morphism $L_z \in \mathcal{M}_\A $ is primitive if and only if each letter from $\A $ occurs in $z$.
 Any primitive morphism in $\mathcal{M}_\A $ has only one  fixed point and this fixed point is a standard Arnoux--Rauzy word.
 On the other hand, if a standard Arnoux--Rauzy word is fixed by a primitive substitution, then it is fixed by a primitive morphism from the monoid $\mathcal{M}_\A $. 

\begin{example} \label{Example_Tribonacci}
Let us consider the Tribonacci word $\uu_\tau = abacabaabacababacabaa \cdots$ ---  the fixed point of the morphism $\tau: a \mapsto ab, b \mapsto ac, c \mapsto a$.
The word $\uu_\tau$ is a standard Arnoux-Rauzy word  over $\{a,b,c\}$ and it is fixed also by the morphism $\tau^3$.
It is easy to check that $\tau^3= L_{abc}$ and thus the Tribonacci word is fixed by a substitution from $\mathcal{M}_\A$.
\end{example}

K. Medková in \cite{Med19} studies derived words of Arnoux--Rauzy words.
She considers all Arnoux--Rauzy (not only standard) words, but she describes derived words only to prefixes of infinite words.
To quote a consequence of one of her results we need to recall the cyclic shift operation on ${\A}^*$:
$$ {\rm cyc}(z_1z_2\cdots z_n)= z_nz_1\cdots z_{n-1}.
$$
\begin{prop}[Theorem 24 in \cite{Med19}]\label{katka}
 Let $L_z\in \mathcal{M}_\A, z \in \A^*, $ be a primitive morphism and  $\uu$ be its fixed point.
If  $w$ is a prefix of $\uu$, then there exists  $k\in \{1,2,\ldots,|z|\}$ such that  $\dd_{\bf \mathbf{u}}(w)$ is fixed (up to a permutation of letters) by $L_{{\rm cyc}^k(z)}$.
In particular, the word $\dd_{\bf \mathbf{u}}(w)$ is a standard Arnoux--Rauzy word.
\end{prop}

\begin{thm}
Let $z$ be a word  in  $\A^*$  such that  each letter $a\in \A$ occurs in $z$ at least once.
The set $$M=\bigl\{L_{\rm cyc^k(z)} \colon k\in \{1,2,\ldots,|z|\}\bigr\}$$ is closed under derivation.
\end{thm}

\begin{proof}
Let $\vv$ be a fixed point of $L_{v }$ with   $v={\rm cyc}^k(z)$ for some $k\in \{1,2,\ldots,|z|\}$.
Since  $z$ contains each letter from $\A$,  the word $v$  contains all letters form $\A$  as well and thus  $L_{v }$ is primitive.

As  $\vv$ is a standard Arnoux--Rauzy word, each its bispecial factor  is a prefix of $\vv$.
By \Cref{claim3}, only derived words to prefixes have to be considered.
By \Cref{katka}, each such derived word is fixed (up to a permutation of letters) by a morphism $L_{{\rm cyc}^k(v)}$ for some $
i \in \{1,2,\ldots,|v|\}$.
Obviously, this morphism belongs to $M$.
\end{proof}

\begin{example}
If we apply the previous theorem to the ternary word $abc$, we obtain that the set  $M =\{L_{abc},L_{bca},L_{cab} \}$ is closed under derivation.
Nevertheless, all the 3 morphisms in $M$ fix (up to a permutation of letters)  the same word, namely the Tribonacci word.
This word is fixed by the substitution $\tau$ given in \Cref{Example_Tribonacci}.
Therefore, the set   $\{\tau \}$ is closed under derivation as well.
\end{example}

\section{Example 2: The period doubling morphism} \label{sec:ex2}

The aim of this section is to show that the period doubling substitution $\psi$  determined by $\psi(a)=ab$ and $ \psi(b)=aa$ is closeable under derivation.
For this purpose, we first define the two following substitutions:
\begin{equation} \label{eq:def_xinu}
\nu: \begin{cases}
0 \mapsto 01, \\ 1 \mapsto 02020101, \\ 2 \mapsto 0202,
\end{cases}
\quad \text{ and } \quad
\xi: \begin{cases}
0 \mapsto 011, \\
1 \mapsto 0.
\end{cases}
\end{equation}

Next, we deduce  several auxiliary statements which help us to prove the following main theorem.

	\begin{thm}\label{theorem}
		The sets $ \{\psi, \xi, \nu\} $ and $ \{\xi, \nu\} $ are closed under derivation.
	\end{thm}
	First, we focus on the derived words  of the fixed point $ \mathbf{z} = abaaabababaaabaaabaa\dotsm $ of the substitution $ \psi $.
The following properties are immediate:
\begin{itemize}
\item  $bb\notin \L(\mathbf{z})$. If $a^i \in \L(\mathbf{z})$,  then $i\leq 3$.
\item  $a$ and  $aa$ are  bispecial factors of $\mathbf{z}$.
\item  Any bispecial factor of length more than 2 has a prefix $ab$ and a suffix $ba$.
\item  The longest common prefix of $\psi(a)$ and $\psi(b)$  is the letter $a$; the longest common suffix of $\psi(a)$ and $\psi(b)$  is the empty word. It implies that $\Phi(v) \vcentcolon =\psi(v)a$ is bispecial whenever $v$ is bispecial.
\end{itemize}

The converse of the very last property also holds (if $\Phi(v)$ is not too short):

\begin{prop}\label{vztah bispecialu}
		Let $ w $ be a non-empty bispecial factor of $ \mathbf{z} $ such that $ w\neq a $ and $ w\neq aa $.
        There exists a bispecial factor $ v $ such that $ \Phi(v) = w $.
\end{prop}

\begin{proof}
As mentioned before, the bispecial factor  $ w $ has a suffix $ ba $ and a prefix $ab$.
Hence, there exists a factor $ v $ such that $\Phi(v) = \psi(v)a = w $ and $ a $ is both a prefix and a suffix of $ v $.
It remains to show that $ v $ is bispecial.
If it is not right special, then $ v $ is followed only by $ a $ or $ b $.
But then $ w $ is followed only by $ b $ or $ a $, respectively, since $ \psi(va) = \psi(v)ab $ and $ \psi(vb) = \psi(v)aa $.
Thus, $w$ is right special.
Similarly, $ v $ is left special, and therefore bispecial.
\end{proof}

As the fixed point $ \mathbf{z}$ has a bispecial factor $aa$ which is not a prefix of  $ \mathbf{z}$,  the description of  derived words to non-prefix factors is more complicated than in  the case of a fixed point of a standard episturmian morphism.
The following notion will be very useful for this purpose.

	\begin{defi}\label{def ancestor}
		Let $ w $ be a non-empty factor of a fixed point $ \mathbf{x} $ of a substitution $ \varphi $.
        Suppose there exist words $ y, y' $ and $ u = u_1u_2\dotsm u_n $ such that $ ywy' = \varphi(u) $, $ |y| < |\varphi(u_1)| $, $ |y'| < |\varphi(u_n)| $, and $u \in \mathcal{L}(\mathbf{x)}$.
        If there is exactly one occurrence of $ w $ in $ \varphi(u) $, then we call $ u $ an \emph{ancestor} of $ w $.
        The set of all ancestors of $ w $ is denoted by~$ A(w) $.
        If there are more occurrences of $ w $ in $ \varphi(u) $, then we say $ w $ \emph{allows an ambiguous ancestor}.
	\end{defi}

	\begin{example}
		Given the fixed point $ \mathbf{z} = abaaabababaaabaaabaa\dotsm $ of the period doubling  substitution $ \psi $, the set of all ancestors of the factor $ aa $ is $ A(aa) = \{b\} $ because $ \psi(b) = aa $ and $ y = \varepsilon, y' = \varepsilon $.
        Since $ \psi(ba) = aaab $, $ y = a $, $ y' = b $ and there are two occurrences of $ aa $ in $ \psi(ba) $, the factor $aa$ allows an ambiguous ancestor.
        The prefix $ aba $ has two ancestors $ aa $ and $ ab $ and it does not allow an ambiguous ancestor.
	\end{example}

	\begin{prop}\label{return words of prefix with unique ancestor}
		Let $ \mathbf{x} $ be a fixed point of an injective substitution $ \varphi $ and $ w $ be a factor of $ \mathbf{x} $ with a unique ancestor $ u $.
        Assume $ w $ does not allow an ambiguous ancestor.
        We have $ \mathbf{d}_{\mathbf{x}}(w) = \mathbf{d}_{\mathbf{x}}(u) $.
	\end{prop}

	\begin{proof}
		The infinite word $ \mathbf{x} $ can be written as $ \mathbf{x} = zr_{i_0}r_{i_1}r_{i_2}\dotsm $, where $ r_{i_j} \in \mathcal{R}_{\mathbf{x}}(u) $ for all $ j \in \mathbb{N}_{0} $.
        If $ u $ is a prefix, then $ z = \varepsilon $.
        By the definition of a return word, $ u $ is a prefix of the word $ r_{i_k}u\dotsm $ for all $ k \in \mathbb{N}_{0} $.
        Since $ u $ is a unique ancestor of $ w $ and $ w $ does not allow an ambiguous ancestor, there are exactly two occurrences of $ w $ in $ \varphi(r_{i_k})\varphi(u) $.
        Let $ \varphi(u) = y w y' $.
		\begin{figure}[H]
			\centering
			\includegraphics[width=0.9\textwidth]{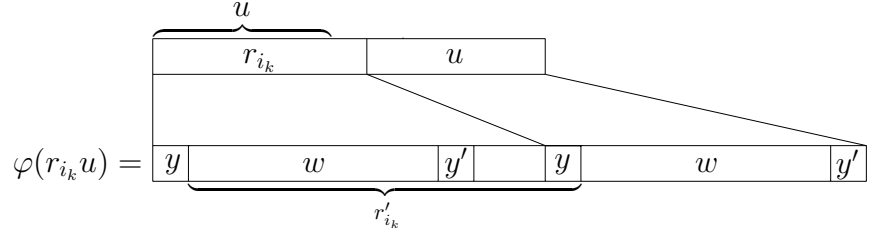}
			\caption{An illustration of $ r_{i_k} u $ and $\varphi(r_{i_k} u)$ in the proof of \Cref{return words of prefix with unique ancestor}.}
			\label{konecne korektni obrazek}
		\end{figure}
        If we define $ r'_{i_k} := y^{-1} \varphi(r_{i_k})y $ as in \Cref{konecne korektni obrazek}, then $ r'_{i_k} \in \mathcal{R}_{\mathbf{x}}(w) $ for all $ k \in \mathbb{N}_{0} $ and we have
		\begin{align*}
			\mathbf{x} &= \varphi(\mathbf{x}) = \varphi(z)\varphi(r_{i_0})\varphi(r_{i_1})\varphi(r_{i_2})\dotsm =\\ &=\underbrace{\varphi(z)y}_{:= z'}\underbrace{y^{(-1)}\varphi(r_{i_0})y}_{r'_{i_0}}\underbrace{y^{(-1)}\varphi(r_{i_1})y}_{r'_{i_1}}\underbrace{y^{(-1)}\varphi(r_{i_2})y}_{_{r'_{i_2}}}y^{(-1)}\dotsm = z'r'_{i_0}r'_{i_1}r'_{i_2}\dotsm.
		\end{align*}
		The derived words of $ u $ and $ w $ are both $ i_0i_1i_2\dotsm $.
	\end{proof}

\begin{lem}\label{stejne} Let $v$ be a non-empty bispecial factor of  the fixed point $ \mathbf{z} $  of the period doubling substitution $\psi$.
We have $ \mathbf{d}_{\mathbf{z}}(\Phi(v)) = \mathbf{d}_{\mathbf{z}}(v) $.
\end{lem}
	\begin{proof}
Since $v$ is bispecial, the word $a$ is a suffix of $v$  and thus  $ \psi(v) $  has a suffix $b$.
It implies that $ \psi(v) $ is not right special.
Therefore $ \mathbf{d}_{\mathbf{z}}(\psi(v)) = \mathbf{d}_{\mathbf{z}}(\psi(v)a) = \mathbf{d}_{\mathbf{z}}(w) $ with $w= \Phi(v) = \psi(v)a $.

The word $ v $ is surely an ancestor of $ \psi(v) $.
We show that it is the only ancestor.
Suppose there is another ancestor $t$ with $t \neq v$.
Since $ \psi $ is injective, there exist $ y,y' \neq \varepsilon $ such that $ y\psi(v)y' = \psi(t) $.
It follows that $ y$ and  $y' $ are both letters.
Thus, the last letter of $ \psi(v) $ is the first letter of $ \psi(a) $ or $ \psi(b) $ which is in both cases the letter $ a $ --- a contradiction.
Therefore $ A(\psi(v)) = \{v\} $ and it is not difficult to verify that $ \psi(v) $ does not allow an ambiguous ancestor when it contains at least one letter $ b $.
By \Cref{return words of prefix with unique ancestor} we have $			\mathbf{d}_{\mathbf{z}}(v) = \mathbf{d}_{\mathbf{z}}(\psi(v)) = \mathbf{d}_{\mathbf{z}}(w)$.
\end{proof}

	\begin{prop}\label{zuzeni na 2 faktory}
	 If $ w $ is a non-empty factor of $ \mathbf{z} $, then $ \mathbf{d}_{\mathbf{z}}(w) = \mathbf{d}_{\mathbf{z}}(a) $ or $ \mathbf{d}_{\mathbf{z}}(w) = \mathbf{d}_{\mathbf{z}}(aa) $.
     If $ w $ is a non-empty  prefix  of $ \mathbf{z} $, then $ \mathbf{d}_{\mathbf{z}}(w) = \mathbf{d}_{\mathbf{z}}(a) $.
	\end{prop}

	\begin{proof}
    By \Cref{claim1}  we have to describe the derived words to right special prefixes and to bispecial factors only.
 First assume that $w$ is a bispecial factor of $ \mathbf{z} $.
By \Cref{vztah bispecialu}, the factor $w$ can be obtained by iteration of the mapping   $\Phi(v)= \psi(v)a$  starting  from the two initial bispecial factors $a$ and $aa$ (in fact, this a special case of a general construction of bispecial factors from \cite{Kl12}).
By \Cref{stejne}, $ \mathbf{d}_{\mathbf{z}}(w)$ equals to  $\mathbf{d}_{\mathbf{z}}(a) $ or to $\mathbf{d}_{\mathbf{z}}(aa) $.

Now assume that $w$ is a right special prefix of $ \mathbf{z} $.
As the initial bispecial factor $a$ is a prefix of $ \mathbf{z} $,  the bispecial factor  $\Phi^k(a)$ is a prefix of $ \mathbf{z} $ for each $k\in \N$.
Therefore, any right special prefix $w$ of $\mathbf{z}$ is left special as well.
More specifically, any right special prefix of $\mathbf{z}$ equals to
$\Phi^k(a)$  for some  $k\in \N$ and by \Cref{stejne},   $ \mathbf{d}_{\mathbf{z}}(w) = \mathbf{d}_{\mathbf{z}}(\Phi^k(a)) = \mathbf{d}_{\mathbf{z}}(a)$.
	\end{proof}

Now we show that both derived words to a factor of ${\mathbf{z}}$ are fixed by primitive substitutions.
We exploit the following simple tool.
 \begin{obs}\label{little lemma}
		Let $ \mathbf{v} $ be a fixed point of a morphism $ \gamma $ and let $ \mathbf{u} = \alpha(\mathbf{v}) $ where $ \alpha $ is a morphism. If there exists a morphism $ \beta $ such that $ \alpha \gamma = \beta \alpha $, then $ \mathbf{u} $ is fixed by $ \beta $.
	\end{obs}
 \begin{proof} $ \beta (\mathbf{u}) = \beta \alpha(\mathbf{v}) = \alpha \gamma(\mathbf{\mathbf{v}}) = \alpha(\mathbf{v}) = \mathbf{u}. $
 \end{proof}
	\begin{prop}\label{der k a-aa}
		The derived word $ \mathbf{d}_{\mathbf{z}}(a) $ is fixed by $ \xi $ and the derived word $ \mathbf{d}_{\mathbf{z}}(aa) $ is fixed by $ \nu $ (where $\xi$ and $\nu$ are defined in \eqref{eq:def_xinu}).
	\end{prop}

	\begin{proof}
In \Cref{example factor a} above, we show that the derived word $ \mathbf{d}_{\mathbf{z}}(a) $ is fixed by the substitution  $\xi $.

It remains to consider $ \mathbf{d}_{\mathbf{z}}(aa) $.
As $abaa$ is the shortest prefix of ${\mathbf{z}}$ containing the bispecial factor $aa$, we can use the construction from the proof of \Cref{der_of_factor} to find a morphism $ \alpha $ such that $ \mathbf{d}_{\mathbf{z}}(aa) = \alpha(\mathbf{d}_{\mathbf{z}}(abaa)) $.
In our  case $p=ab $ and  $w=aa$.
 According to \Cref{zuzeni na 2 faktory}, the derived word $ \mathbf{d}_{\mathbf{z}}(abaa) $ is fixed by $ \xi $ since $ \mathbf{d}_{\mathbf{z}}(a) $ is fixed by $ \xi $.
 Thus, $ \mathbf{d}_{\mathbf{z}}(abaa) $ is over a binary alphabet, and so the prefix $ abaa $ has exactly two return words, say $ r_0$ and $r_1 $.
 These two return words  can be found  in the prefix  of  $ \mathbf{z} $  of length $16$.
 They are
		$$ r_0 = abaaabab \quad \text{ and } \quad  r_1 = abaa.$$
		It follows from the proof of \Cref{der_of_factor} that $(ab)^{-1} r_0ab $ and $(ab)^{-1} r_1ab $ can be written as a concatenation of return words to $ aa $.
        Specifically, $ r'_0 = a, r'_1 = aababab, r'_2 = aab $ are return words of $ aa $ and $ (ab)^{-1}r_0ab = r'_0r'_1$ and   $ (ab)^{-1}r_1ab = r'_0r'_2 $.
        Hence, according to this claim we have
		\begin{align*}
			\alpha(0) &= 01,\\
			\alpha(1) &= 02.
		\end{align*}

	Note that since $ \mathbf{d}_{\mathbf{z}}(abaa) $ is fixed by $ \xi $, it is also fixed by $ \xi^2 $. By \Cref{little lemma}, if the substitution $ \nu $ satisfies $ \alpha \xi^2 = \nu \alpha $, the proof is finished. This is very easy to verify:
		\[
        \begin{split}
			\alpha \xi^2 (0) &= \alpha(01100) = 0102020101\\
			\nu \alpha (0) &= \nu(01) = 0102020101\\
			\alpha \xi^2 (1) &= \alpha(011) = 010202\\
			\nu \alpha (1) &= \nu(02) = 010202. \qedhere
		\end{split}
        \]
	\end{proof}

	\begin{remark}
		The derived word $ \mathbf{d}_{\mathbf{z}}(aa) $ is also fixed by the  morphism
		\begin{align*}
			 \eta(0) &= \varepsilon \\
			 \eta(1) &= 010202 \\
			 \eta(2) &= 01.
		\end{align*}
		A proof is the same as the proof of \Cref{der k a-aa}, but at the end we have to verify the equality $ \alpha \xi = \eta \alpha $.
        The reason why we prefer  $ \nu $ to  $ \eta $ is  that $ \eta $ is an erasing non-primitive morphism.
	\end{remark}

	\begin{coro}\label{stezejni dusledek}
		If $ w $ is a non-empty factor of $ \mathbf{z} $, then $ \mathbf{d}_{\mathbf{z}}(w) $ is fixed by $ \xi $ or $ \nu $.
	\end{coro}

	\begin{proof}
		The corollary follows from \Cref{zuzeni na 2 faktory,der k a-aa}.
	\end{proof}

We conclude this section by the proof of our main result.
For this purpose we need one more ingredient.
It is a modification of  Proposition 6, Item 5 from \cite{Durand98}.
Its proof is almost identical with the proof of the original statement and thus we omit it.

	\begin{lem}\label{Durand lemma}
		Let $ \mathbf{u} $ be a uniformly recurrent word and let $ w $ be its factor.
        Set $ \mathbf{v} = \mathbf{d}_{\mathbf{u}}(w) $.
        For a factor $ x $ of $ \mathbf{v} $, there exists a factor $ y $ of $ \mathbf{u} $ such that $  \mathbf{d}_{\mathbf{v}}(x) = \mathbf{d}_{\mathbf{u}}(y) $.
	\end{lem}

\begin{proof}[Proof of \Cref{theorem}]
 Let $ \mathbf{v} $ be a fixed point of the  primitive substitution $ \xi $  and $x$  be a factor of $ \mathbf{v} $.
 By \Cref{der k a-aa}, we have  $ \mathbf{v} = \mathbf{d}_{\mathbf{z}}(a)$.
 By \Cref{Durand lemma}, there exists a factor $y$ in  ${\mathbf{z}}$ such that   $\mathbf{d}_{\mathbf{v}}(x) = \mathbf{d}_{\mathbf{z}}(y) $.
 \Cref{zuzeni na 2 faktory} implies that $\mathbf{d}_{\mathbf{v}}(x)$ equals  $\mathbf{d}_{\mathbf{z}}(a)$ or $\mathbf{d}_{\mathbf{z}}(aa)$.
 Therefore, $\mathbf{d}_{\mathbf{v}}(x)$ is fixed by $\xi$ or $\nu$.

 The same reasoning gives that the derived word to any factor of the fixed point of $\nu$ is fixed by $\xi$ or by $\nu$.
 By \Cref{DEF}, the set $\{\nu, \xi\}$  is closed under derivation.

 As $\mathbf{d}_{\mathbf{z}}(\varepsilon) = {\mathbf{z}}$  and the derived word to  any non-empty factor of ${\mathbf{z}}$ is fixed by $\xi$ or by $\nu$,
 the set $\{\nu, \xi, \psi\}$ is also closed under derivation.
\end{proof}


\section*{Acknowledgments}

This work was supported by the Ministry of Education, Youth and Sports of the Czech Republic, project no. CZ.02.1.01/0.0/0.0/16\_019/0000778.
We also acknowledge financial support of the Grant Agency of the Czech Technical University in Prague, grant No. SGS17/193/OHK4/3T/14.
We thank to Michel Dekking for attracting our attention to the article~\cite{HuangWen17}.


\end{document}